\documentclass[reqno, 12pt]{amsart}
\usepackage{verbatim}
\usepackage{amssymb}
\usepackage{enumerate}
\usepackage[active]{srcltx}
\numberwithin{equation}{section}

\usepackage{t1enc}

\newtheorem{theorem}{Theorem}[section]

\newtheorem{lemma}[theorem]{Lemma}

\newtheorem{claim}{Claim}[theorem]

 \usepackage{color}

\theoremstyle{definition}

\newcommand{\setm}{\setminus}

\newcommand{\subs}{\subset}

\def\<{\left\langle}
\def\>{\right\rangle}
\def\br#1;#2;{\bigl[ {#1} \bigr]^ {#2} }

\def\aint#1;#2;{A(#1,#2)}
\def\aintd#1;#2;{\dot A(#1,#2)}

\newcommand{\cf}{\operatorname{cf}}

 \author[I. Juh\'asz]{Istv\'an Juh\'asz}
 \address
       { Alfréd Rényi Institute of Mathematics, Hungarian Academy of Sciences
 }
 \email{juhasz@renyi.hu}
 \author[L. Soukup]{Lajos Soukup}
 \address
       { Alfréd Rényi Institute of Mathematics, Hungarian Academy of Sciences}
 \email{soukup@renyi.hu}
 \author[Z. Szentmikl\'ossy]{Zolt\'an Szentmikl\'ossy}
 \address{E\"otv\"os University of Budapest}
 \email{szentmiklossyz@gmail.com}
 \title[Resolvability of pseudocompact spaces]
    {Coloring Cantor sets and resolvability of pseudocompact spaces}
 \thanks{The research on and preparation of this paper was
 supported by  NKFIH grant no. K113047.}
\begin{document}

\begin{abstract}
Let us denote by $\Phi(\lambda,\mu)$ the statement that $\mathbb{B}(\lambda) = D(\lambda)^\omega$,
i.e. the Baire space of weight $\lambda$, has a coloring with $\mu$ colors such
that every homeomorphic copy of the Cantor set $\mathbb{C}$ in $\mathbb{B}(\lambda)$ picks up all the $\mu$ colors.

We call a space $X\,$ {\em $\pi$-regular} if it is Hausdorff and for every non-empty open set $U$ in $X$ there is a
non-empty open set $V$ such that $\overline{V} \subset U$. We recall that a space $X$ is called {\em feebly compact}
if every locally finite collection of open sets in $X$ is finite. A Tychonov space is pseudocompact iff it is
feebly compact.

The main result of this paper is the following.

\medskip
\noindent {\bf Theorem.}
{\em Let $X$ be a crowded feebly compact $\pi$-regular space
and $\mu$ be a fixed (finite or infinite) cardinal.
If $\Phi(\lambda,\mu)$ holds for all $\lambda < \widehat{c}(X)$ then $X$ is $\mu$-resolvable,
i.e. contains $\mu$ pairwise disjoint dense subsets. (Here $\widehat{c}(X)$ is the smallest
cardinal $\kappa$ such that $X$ {\em does not contain} $\kappa$ many pairwise disjoint open sets.)}
\medskip

This significantly improves earlier results of van Mill \cite{vM}, resp. Ortiz-Castillo and Tomita \cite{OCT}.
\end{abstract}

\dedicatory{Dedicated to the memory of our old friend Bohus Balcar.}

\maketitle

\section{Introduction}

It is well-known that any compact Hausdorff space $X$ is maximally
resolvable, i.e $\Delta(X)$-resolvable, where $\Delta(X)$ is the smallest cardinality
of a non-empty open set in $X$. The question if this is also true for countably compact
regular spaces, however is completely open.

Pytkeev proved in \cite{Py} that crowded
countably compact regular spaces are $\omega_1$-resolvable and this result was extended
in \cite{JSSz} to crowded countably compact $\pi$-regular spaces.
$X\,$ is {\em $\pi$-regular} if it is Hausdorff and for every non-empty open set $U$ in $X$ there is a
non-empty open set $V$ such that $\overline{V} \subset U$.
Since any $\pi$-regular crowded countably compact space $X$
satisfies $\Delta(X) \ge \mathfrak{c}$, the natural question was raised there if $\omega_1$
can at least be replaced here by $\mathfrak{c}$, in ZFC.

If one weakens countably compact to pseudocompact, an even tougher problem seems to arise:
It is still unknown in ZFC whether all crowded pseudocompact spaces are at least
(2-)resolvable. We recall that Hewitt in \cite{H} defined pseudocompact spaces as those
{\em Tychonov} spaces on which every continuous real valued function is bounded.

Marde\v si\' c and Papi\' c defined in \cite{MaP} {\em feebly compact} spaces by having the
property that every locally finite collection of open sets in them is finite and showed
that a Tychonov space is pseudocompact iff it is feebly compact.
In his survey paper \cite{Pa} on problems on resolvability, in addition to the problem if
pseudocompact spaces are resolvable, Pavlov also asked if crowded regular feebly compact spaces
are resolvable.

The first significant advance on the problem if pseudocompact spaces are resolvable was done
by van Mill in \cite{vM} where it was shown that any crowded pseudocompact space $X$ is
$\mathfrak{c}$-resolvable, provided that $c(X) = \omega$, i.e. $X$ is CCC. Of course, this
also means that crowded CCC countably compact Tychonov spaces are $\mathfrak{c}$-resolvable.

At the 2016 TOPOSYM in Prague, Ortiz-Castillo and Tomita announced the following (partial)
improvement of van Mill's result: Every crowded pseudocompact space $X$ that satisfies
$c(X) \le \mathfrak{c}$ is resolvable, see \cite{OCT}. (It is not clear from \cite{OCT} if their proof
actually gives $\mathfrak{c}$-resolvability of $X$.)

\medskip

Our main result improves those of van Mill and Ortiz-Castillo-Tomita and, in addition, also gives
a partial affirmative answer to Pavlov's second question mentioned above. To help to formulate it,
we introduce some notation and terminology.

We denote by  $\mathbb{B}(\lambda)$ the $\omega$th power of the discrete space on $\lambda$,
i.e. the Baire space of weight $\lambda$. Moreover, given some cardinals $\lambda$ and $\mu$,
We denote by $\Phi(\lambda,\mu)$ the statement that $\mathbb{B}(\lambda)$
has a coloring with $\mu$ colors such
that every homeomorphic copy of the Cantor set $\mathbb{C}$ in $\mathbb{B}(\lambda)$ picks up
all the $\mu$ colors. Using the arrow notation of partition calculus, $\Phi(\lambda,\mu)$ can be written as
$\,\mathbb{B}(\lambda) \nrightarrow [\mathbb{C}]^1_\mu$.
Finally, we recall from \cite{J} that for any space $X$ we denote by $\widehat{c}(X)$ the smallest
cardinal $\kappa$ such that $X$ does not contain $\kappa$ many pairwise disjoint open sets.

\begin{theorem}\label{tm:main}
Let $X$ be a crowded feebly compact $\pi$-regular space and $\mu$ be a fixed cardinal.
If $\Phi(\lambda,\mu)$ holds for all $\lambda < \widehat{c}(X)$ then $X$ is $\mu$-resolvable.
\end{theorem}

It follows from results proved in  \cite{HJS} that if $\varrho < \mathfrak{c}$ is a regular cardinal then
$\Phi(\lambda, \mathfrak{c})$ holds for all $\lambda < \mathfrak{c}^{+\varrho}$.
Consequently any crowded feebly compact $\pi$-regular space
is $\mathfrak{c}$-resolvable, provided that $c(X) <  \mathfrak{c}^{+\omega}$,
or even if $c(X) <  \mathfrak{c}^{+\omega_1}$ when CH fails.

Another consequence of those results, see also \cite{W},
is that $\Phi(\lambda, \mathfrak{c})$ holds for all cardinals $\lambda$ provided
that for every singular cardinal $\nu$ of countable cofinality both $\nu^\omega = \nu^+$
and $\Box_\nu$ are valid. This in turn implies the consistency of the $\mathfrak{c}$-resolvability
of all crowded feebly compact $\pi$-regular spaces, moreover that the consistency of the negation of this
requires large cardinals.

Finally we mention the following unpublished result of W. Weiss \cite{W1}: If $\omega_1$ Cohen reals are
added to any ground model $V$, then the generic extension $V^{Fn(\omega_1, 2)}$ satisfies
$\Phi(\lambda, \omega_1)$ for all cardinals $\lambda$.

\section{The proof of Theorem \ref{tm:main}}

Let us fix the crowded feebly compact $\pi$-regular space $X$ and the cardinal $\mu \le \mathfrak{c}$
such that $\Phi(\lambda,\mu)$ holds for all $\lambda < \kappa = \widehat{c}(X)$. Note that then $\kappa$
is an uncountable regular cardinal, see e.g. 4.1 of \cite{J}. The topology of $X$ is denoted by $\tau$
and we write $\tau^+ = \tau \setm \{\emptyset\}$.

We call a sequence $\overrightarrow{U} = \langle U_n : n < \omega \rangle$ of members of $\tau^+$
{\em strongly decreasing} (in short: SD) if $\overline{U_{n+1}} \subs U_n$ holds for every $n < \omega$.
We shall denote by $SD(X)$ the family of all strongly decreasing sequences.

For  $\overrightarrow{U} = \langle U_n : n < \omega \rangle \in SD(X)$ we put:
$\cap\overrightarrow{U} = \bigcap \{U_n : n < \omega \}$. Clearly, $\cap\overrightarrow{U}$
is always closed in $X$, moreover it is non-empty because $X$ is feebly compact,
see  Lemma \ref{lm:ne} below.
(It can be shown that the converse of this is also true if $X$ is $\pi$-regular.)
Finally, for any $\,\overrightarrow{U}\in SD(X)$ we shall denote by $\partial\overrightarrow{U}$
the {\em boundary} of $\cap\overrightarrow{U}$.

The following simple lemma will play a crucial role in our proof.

\begin{lemma}\label{lm:ne}
Assume that $\overrightarrow{U} = \langle U_n : n < \omega \rangle \in SD(X)$ and $V \in \tau$
are such that for every $n < \omega$ we have $$V \cap (U_n\,\setm\,\cap\overrightarrow{U}) \ne \emptyset.$$
Then $\overline{V} \cap \partial\overrightarrow{U} \ne \emptyset$ as well.
\end{lemma}

\begin{proof}
Let us put $W_n = V \cap (U_n\,\setm\,\cap\overrightarrow{U})$. Then $\langle W_n : n < \omega \rangle$ is a
decreasing sequence of non-empty open sets in $X$, hence there is a point $x$ with $x \in \overline{W_n}$
for infinitely many, hence all, $n \in \omega$ because $X$ is feebly compact. But then we have $x \in \overline{W_n} \subs \overline{V}$
and also $$x \in \cap\overrightarrow{U} \setm Int(\cap\overrightarrow{U}) = \partial\overrightarrow{U}$$
because $W_n \cap \cap\overrightarrow{U} = \emptyset$ implies $\overline{W_n} \cap Int(\cap\overrightarrow{U}) = \emptyset$.
\end{proof}

Next, using that $X$ is both crowded and $\pi$-regular, we fix for every open $U \in \tau^+$ a {\em maximal} family $\mathcal{S}(U) \subs \tau^+$
with $|\mathcal{S}(U)| > 1$ such that
\begin{enumerate}[(i)]
\item $V \in \mathcal{S}(U)$ implies $\overline{V} \subs U$;

\smallskip

\item if $V,W \in \mathcal{S}(U)$ are distinct then $\overline{V} \cap \overline{W} = \emptyset$.
\end{enumerate}
Clearly then the union of $\mathcal{S}(U)$ is dense in $U$.

Using the operation $\mathcal{S}(U)$ we now define a tree $T$ whose nodes are members of $\tau^+$
and the tree ordering is reverse inclusion
in the following natural way. The levels $T_\alpha$ are defined by transfinite recursion,
starting with $T_0 = \{X\}$.

If $T_\alpha$ has been defined then we put $$T_{\alpha+1} = \bigcup\{\mathcal{S}(U) : U \in T_\alpha\}$$
where, of course, the immediate successors of any $U \in T_\alpha$ are the members of $\mathcal{S}(U)$.

Finally, if $\alpha$ is a limit ordinal and $T\upharpoonright\alpha = \bigcup_{\beta < \alpha} T_\beta$
has been defined then $T_\alpha$ consists of all sets of the form
$Int(\cap B)$, where $B$ is any cofinal branch of $T\upharpoonright\alpha$ such that
$Int(\cap B) \ne \emptyset$.

Let $\theta$ be the height of $T$, i.e. the smallest ordinal for which $T_\theta = \emptyset$.
Clearly, $\theta$ is a limit ordinal. Note that every branch $B = \{U_\alpha : \alpha < \gamma\}$ of $T$
(where, of course, $U_\alpha \in T_\alpha$) has cardinality $< \kappa$ because
$$\{U_\alpha \setm U_{\alpha+1} : \alpha < \gamma\} \subs \tau^+$$ is a pairwise disjoint
collection of open sets.
This implies that $\theta \le \kappa$.

It is also obvious that every antichain of $T$, in particular every level $T_\alpha$,
also has cardinality less than $\kappa$, hence if $\theta = \kappa$
then $T$ is a $\kappa$-Suslin tree, however we shall not need this fact.

Let us denote by $E$ the set of all ordinals $\alpha \le \theta$ with $\cf(\alpha) = \omega$.
Then for every $\alpha \in E$ we may fix a strictly increasing sequence
$\langle\alpha_n : n < \omega\rangle$  with $\sup\{\alpha_n : n < \omega\} = \alpha$.
We then have $$\kappa_\alpha = \sup\{|T_{\alpha_n}| : n< \omega\} < \kappa$$
because $\kappa$ is an uncountable regular cardinal and
every level of $T$ has size $ < \kappa$.

Consequently, for every $\alpha \in E$ there is by our assumptions a coloring
$b_\alpha : \mathbb{B}(\kappa_\alpha) \to \mu$ such that every homeomorphic copy of
the Cantor set in $\mathbb{B}(\kappa_\alpha)$ picks up all colors $\nu \in \mu$.

To simplify our notation, using that $|T_{\alpha_n}| \le \kappa_\alpha$, we consider
each  level $T_{\alpha_n}$ of $T$ embedded in $\kappa_\alpha$, hence their product
$\prod_{n < \omega}T_{\alpha_n}$ is embedded in ${\kappa_\alpha}^\omega = \mathbb{B}(\kappa_\alpha)$.
Then it makes sense to put $$\mathcal{Y}_\alpha = SD(X) \cap \prod_{n < \omega}T_{\alpha_n}.$$
In other words, $\mathcal{Y}_\alpha$ consists of all
strongly decreasing sequences $\overrightarrow{U} = \langle U_n : n < \omega \rangle \in SD(X)$
for which there is a cofinal branch $B$ of $T\upharpoonright\alpha$ such that for every $n < \omega$
we have $B \cap T_{\alpha_n} = \{U_n\}$. Since $\mathcal{Y}_\alpha \subs \mathbb{B}(\kappa_\alpha)$,
the coloring $b_\alpha$ is defined, in particular, on $\mathcal{Y}_\alpha$.

Next we move back to the space $X$ and define $$Y_\alpha = \bigcup \{\partial\overrightarrow{U} : \overrightarrow{U} \in \mathcal{Y}_\alpha\}.$$
It is clear from our construction of the tree $T$ that for distinct $\overrightarrow{U},\,\overrightarrow{V} \in \mathcal{Y}_\alpha$ we have
$\partial\overrightarrow{U} \cap \partial\overrightarrow{V} = \emptyset$. Moreover, if $\alpha$ and $\beta$ are distinct elements of $E$ then
$Y_\alpha \cap Y_\beta = \emptyset$ as well.

Thus we may define for each $\alpha \in E$ the coloring $h_\alpha : Y_\alpha \to \mu$ by putting
$h_\alpha(y) = b_\alpha(\overrightarrow{U})$ for $y \in Y_\alpha$, where $\overrightarrow{U}$ is the unique element of
$\mathcal{Y}_\alpha$ such that $y \in \partial\overrightarrow{U}$. Moreover, $h = \bigcup_{\alpha \in E} h_\alpha$
is a well-defined coloring of $Y = \bigcup_{\alpha \in E} Y_\alpha$ with $\mu$ colors.
Our aim now is to show that $h$ establishes the $\mu$-resolvability of $X$ because $h^{-1}\{\nu\}$ is
dense in $X$ for every $\nu < \mu$. Equivalently, this means that for every open set $U \in \tau^+$
we have $h[U] = \mu$.
Actually, we shall prove the following stronger statement.

\begin{lemma}\label{lm:res}
For every $\,U \in \tau^+$ there is an $\alpha \in E$ such that $$h_\alpha[U] = \mu.$$
\end{lemma}

\begin{proof}
Let us define for every open $U \in \tau^+$ the ordinal $\gamma_U$ as follows:
$$\gamma_U = \min \{\gamma \le \theta : (\cup T_\gamma) \cap U = \emptyset \}.$$
$\gamma_U$ exists because $\cup T_\theta = \emptyset$, moreover
it is clear from our construction of the tree $T$ that $\gamma_U$ is a limit ordinal.

Let us say that $W \in \tau^+$ is {\em good} if for every $V \in \tau^+$ with $V \subs W$
we have $\gamma_V = \gamma_W$. It is clear from our assumptions that for every $U \in \tau^+$
there is a good $W \in \tau^+$ such that $\overline{W} \subs U$. Consequently, it suffices to show that
for every good $W \in \tau^+$ there is an $\alpha \in E$ such that $h_\alpha \big[\,\overline{W}\,\big] = \mu$.

So, consider any good $W \in \tau^+$ and put $\gamma_W = \gamma$. We are going to define a strictly increasing sequence
of ordinals $\beta_n < \gamma$ and an injective map $f : 2^{<\omega} \to T\upharpoonright\gamma$ such that
for every $n < \omega$ and $s \in 2^n$ we have $f(s) \in T_{\beta_n}$ and $f(s) \cap W \ne \emptyset$.
The ordinals $\beta_n$ and the map $f$ will be defined by recursion on $n < \omega$, using the following claim.

\begin{claim}\label{cl:b}
For every $V \in \tau^+$ with $V \subs W$ there is an ordinal $\beta(V) < \gamma$ such that
for any $\beta \in [\beta(V), \gamma)$ we have $$|\{U \in T_\beta : V \cap U \ne \emptyset\}| > 1\,.$$
\end{claim}

Indeed, for every $\beta < \gamma$ there is some $U \in T_\beta$ such that $U \cap V \ne \emptyset$
because $\gamma = \gamma_V$. But if no other member of $T_\beta$ intersects $V$ then $V \subs \overline{U}$
since otherwise, as $W$ is good, $\gamma_{(V \setm \overline{U})} = \gamma$ would result in a contradiction.
Thus if we had cofinally many $\beta < \gamma$ with a unique $U_\beta \in T_\beta$ intersecting $V$,
then there would be a cofinal branch $B$ of $T\upharpoonright\gamma$ with $V \subs \cap B$, contradicting
the choice of $\gamma_V = \gamma$. Thus Claim \ref{cl:b} has been proven.

\medskip

Now we define the ordinals $\beta_n < \gamma$ and the injection $f : 2^{<\omega} \to T\upharpoonright\gamma$
by recursion as follows. First we choose $\beta_0 < \gamma$ arbitrarily and then put $f(\emptyset) = U_0$
where $U_0$ is any member of $T_{\beta_0}$ such that $W \cap U_0 \ne \emptyset$. This is possible because
$\beta_0 < \gamma = \gamma_W$.

If we have already defined $\beta_n$ and $f(s)$ for all $s \in 2^{n}$ appropriately, then we may apply Claim
\ref{cl:b} to each non-empty open subset $W \cap f(s)$ of $W$ with $s \in 2^{n}$ and choose
$\beta_{n+1} < \gamma$ above $\beta_n$ and the finitely many values $\beta(W \cap f(s))$ for $s \in 2^{n}$.
Clearly, this recursion goes through and results in the desired sequence of ordinals $\beta_n$ and injection $f$.

Now, we clearly have $\sup\{\beta_n : n < \omega\} = \alpha \in E$ and we claim that $h_\alpha \big[\,\overline{W}\,\big] = \mu$.
To see this, note first that for every $x \in \mathbb{C} = 2^\omega$ the sequence
$\<f(x\upharpoonright n) : n < \omega \>$ determines a cofinal branch $B_x$ of the cut off tree $T\upharpoonright\alpha$.
This branch $B_x$ then determines the sequence $\overrightarrow{U}_x = \langle U_{x,n} : n < \omega \rangle \in \mathcal{Y}_\alpha$,
where $U_{x,n}$ is the unique element of $B_x$ on the level $T_{\alpha_n}$. The injectivity of $f$ clearly implies that if $x \ne x'$
then $B_x \ne B_{x'}$.

In this way we also obtain a map $F : \mathbb{C} \to \mathcal{Y}_\alpha \subs \mathbb{B}(\kappa_\alpha)$ defined by
$F(x) = \overrightarrow{U}_x$. This map $F$ is clearly injective. Moreover,
$F$ considered as a map from $\mathbb{C}$ to $\mathbb{B}(\kappa_\alpha)$ is also uniformly continuous. Indeed,
for every $n < \omega$ there is $m < \omega$ such that $\alpha_n \le \beta_m$, and then for any $x,x' \in 2^\omega$
with $x\upharpoonright m = x'\upharpoonright m$ we clearly have
$$F(x)\upharpoonright n = \overrightarrow{U}_x \upharpoonright n = \overrightarrow{U}_{x'} \upharpoonright n = F(x')\upharpoonright n .$$
But this means that $F$ actually is a homeomorphism, hence we have
$b_\alpha\big[F[\mathbb{C}]\big] = \mu$.

For every $x \in \mathbb{C}$ we have $F(x) = \overrightarrow{U}_x \in \mathcal{Y}_\alpha \subs SD(X)$ and, by our construction,
for every $n < \omega$ there are $m, k < \omega$ with $\alpha_k > \beta_m > \alpha_n$, and then
$\emptyset \ne W \cap (U_{x,n} \setm U_{x,k}) \subs W \cap (U_{x,n} \setm \cap \overrightarrow{U}_x)$.
Consequently we may apply Lemma \ref{lm:ne} to conclude that $\overline{W} \cap \partial(\overrightarrow{U}_x) \ne \emptyset$.
So, let us pick for each $x \in \mathbb{C}$ a point $y_x \in \overline{W} \cap \partial(\overrightarrow{U}_x)$ and recall that we have
$h_\alpha(y_x) = b_\alpha(\overrightarrow{U}_x) = b_\alpha(F(x))$. Consequently, we indeed have
$h_\alpha \big[\,\overline{W}\,\big] = b_\alpha[F[\mathbb{C}]] = \mu$.
This completes the proof of Lemma \ref{lm:res} and with that the proof of Theorem \ref{tm:main}.
\end{proof}


\end{document}